\documentclass{amsart}
\usepackage{amsmath,amsthm,amsfonts,amssymb,latexsym,mathrsfs,graphicx,tikz,array}
\usepackage{subcaption}
\usepackage{hyperref}

\usepackage{xpatch}
\makeatletter
\xpatchcmd{\@thm}{\thm@headpunct{.}}{\thm@headpunct{}}{}{}
\makeatother
\usepackage{enumerate}
\usepackage[shortlabels]{enumitem}
\usepackage{color}

\newcolumntype{P}[1]{>{\centering\arraybackslash}p{#1}}

\newtheorem{theorem}{Theorem}[section]

\newtheorem{lemma}[theorem]{Lemma}
\newtheorem{remark}[theorem]{Remark}

\theoremstyle{definition}

\newtheorem*{Thm}{Theorem}
\newtheorem*{Thmm}{Main Theorem}

\newtheorem*{Cor}{Corollary}
\newtheorem{Question}[theorem]{Question}

\newcommand{\Aut}{{\mathrm {Aut}}}

\newcommand{\Syl}{\mathrm{Syl}}

\newcommand{\OO}{\mathrm{o}}

\newcommand{\sg}{|G|}
\newcommand{\nn}{n_{2}(G)}
\newcommand{\nnn}{n_{3}(G)}

\makeatletter
\newcommand*{\rom}[1]{\expandafter\@slowromancap\romannumeral #1@}
\makeatother

\begin{document}
	
	\title{Characterizing some finite groups by the average order}

	\author[A. Zarezadeh]{Ashkan ZareZadeh$^{1}$}
	\author[B. Khosravi]{Behrooz Khosravi$^1$}
	\author[Z. Akhlaghi]{Zeinab Akhlaghi$^{1,2}$}

	\address{$^{1}$ Faculty of Mathematics and Computer Science, Amirkabir University of Technology (Tehran Polytechnic), 15914 Tehran, Iran.}
	\address{$^{2}$ School of Mathematics,
		Institute for Research in Fundamental Science (IPM)
		P.O. Box:19395-5746, Tehran, Iran. }
	\email{\newline \text{(A. Zarezadeh) }ashkanzarezadeh@gmail.com  \newline \text{(B. Khosravi) }khosravibbb@yahoo.com \newline \text{(Z. Akhlaghi) }z\_akhlaghi@aut.ac.ir  }

	\begin{abstract}
		The average order of a finite group $G$ is denoted by $\OO(G)$. In this note, we classify groups whose average orders are less than $\OO(S_4)$, where $S_{4}$ is the symmetric group on four elements. Moreover, we prove that $G\cong S_4$ if and only if $\OO(G)=\OO(S_4)$. As a consequence of our results we give a characterization for some finite groups by the average order. In  \cite[Theorem 1.2]{t},  the groups whose average orders are less than  $\OO(A_4)$ are classified. It is worth mentioning that to get our results we avoid using the main theorems of \cite{t} and our  results leads to  reprove   those theorems.     
		
	\end{abstract}
\thanks{The first author is supported by a grant from IPM (No. 1402200112) }
	
	\maketitle
	
	\section{Introduction}
	
	For a finite group $G$, $\psi(G)$ was first introduced in \cite{1}, which denotes the sum of the element orders of $G$, i.e., $\psi(G)=\sum\limits_{x\in G}{o(x)}$. Later the average order of $G$ was defined as, $\OO(G)=\frac{\psi(G)}{|G|}$.

	At first glance, these quantities are not expected to inform us much about the structure of $G$. As an example, $\psi(A_{4})=\psi(D_{10})=31$, and also     $\OO(D_{12})=\OO(C_{4})=2.75$. If $G$ is a group such that there exist(s) exactly $k$ non-isomorphic groups with average order $\OO(G)$, then we say $G$ is {\it k-recognizable by average order}. A $1$-recognizable group is called {\it characterizable}.
	In  \cite{11}, it is conjectured that $\OO(G)<\OO(A_5)$, guaranties the solvability of  $G$, which turned out to be true when Herzog, Longobardi and Maj proved it in \cite{hlm}. In that paper, they also classified all finite groups with $\OO(G)\le \OO(S_{3})$. Later in \cite{A_5}, they proved that $A_{5}$ is characterizable by average order, meaning that $\OO(G)=\OO(A_{5})$ implies $G\cong A_{5}$. Meanwhile Tărnăuceanu in \cite{t}, classified all finite groups with $\OO(G)\le \OO(A_{4})$, where he stated that if $\OO(G)< \OO(A_{4})$, $G$ would be supersolvable, and $\OO(G)=\OO(A_{4})$ leads to $G\cong A_{4}$. As the main result of this paper, inspired by the above results, we determine the groups whose average orders are less than $2.8$. Note that the bound $2.8$ is close to $\OO(S_4)=\frac{67}{24}\approx 2.791$. As a consequence, we give a new characterization for some finite groups using the average order. As the main result, we prove the following theorem:

	\begin{Thmm}\label{A}
		Let $G$ be a finite group satisfying $\OO(G)<2.8$. Then $G$ is isomorphic to one of the following groups:
		\begin{enumerate}[\itshape(i)]
			
			\item $D_{12}$;
			
			\item $S_4$;
			
			\item $C_3$ or $C_{3}\times C_{3}$;
			
			\item  $C_2^{2k}\rtimes C_3$, for some natural number $k$, which is a Frobenius group;

			\item  a $2$-group of one of the following types, where $E\cong C_{2}^{m}$, for some $m\ge0$:
			
				\begin{enumerate}[\itshape(1)]
					\item $C_{4}$;
					\item $D_{8}\times E$;
					\item an elementary abelian $2$-group;
					\item $D(C_{4}\times C_{4}) \times E$, where $D(C_{4}\times C_{4})$ is the generalized dihedral of $C_{4}\times C_{4}$;
					\item $(D_{8}* D_{8})\times E$, where $(D_{8}* D_{8})$ is the central product of $D_{8}$ and $D_{8}$;
					\item $S(2)\times E$, where $S(2)=\langle c,x_{1}, y_{1}, x_{2}, y_{2}|  c^{2}=x_{i}^{2}=y_{i}^{2}=e$, all pairs of generators commute except $[c,x_{i}]=y_{i} \rangle$.
					
				\end{enumerate}
			\item  $C_3^{k}\rtimes C_2$, for some natural number $k$, which is a Frobenius group.
		\end{enumerate}
	\end{Thmm}

	Determining finite groups whose average orders are less than $2.8$ can also be helpful in studying the density of the quantity $\OO(G)$ which was discussed in \cite{t2}. To prove our results, we do not use \cite[Theorem 1.2]{t} and as a consequence, we pose another proof for those results.
	
   	Throughout this paper, all groups are non-trivial finite groups, for $\theta \in \Aut(G)$, $r(G,\theta)$ denotes $n(\theta)/\sg$, where $n(\theta)$ is the number of the elements mapped to their inverses by $\theta$, $n_{k}(A)$ denotes the number of the elements of order $k$ in $A\subseteq G$, and $\pi(G)$ is the set of all prime divisors of $\sg$, $H(r)$ is the central product of $r$ copies of $D_{8}$, $S(r)=\langle c,x_{1}, y_{1}, x_{2}, y_{2}, ...,x_{r}, y_{r}|  c^{2}=x_{i}^{2}=y_{i}^{2}=e$, all pairs of generators commute except $[c,x_{i}]=y_{i} \rangle$, and $D(A)$ is the generalized dihedral group of abelian group $A$.

	\begin{Cor}\label{B}
		By the notations in the main theorem:
		\begin{enumerate}[\itshape(a)]
		\item if $G$ is an elementary abelian $2$-group, or isomorphic to $D_{8}\times C_{2}^{k}$, for some integer $k\ge 0$, or one of the groups mentioned in Cases $(ii)$, $(iii)$, $(iv)$, $(vi)$, then $G$ is characterizable by average order,
		\item if $G$ is isomorphic to $D_{12}$ or $C_{4}$, then $G$ is $2$-recognizable by average order,
		\item if $G$ is isomorphic to the groups stated in \itshape(4), \itshape(5) or \itshape(6) of Case $(v)$, then $G$ is $3$-recognizable by average order.
		\end{enumerate}
	\end{Cor}

	\section{preliminary results}
	
	The following lemmas are used several times in  the proof of our main results.  The first lemma is applied in  the proofs multiple times without any further reference.\\ 
	
	From the proof of \cite[Lemma 3.1]{hlm}, we have the following lemma:
	\begin{lemma}\label{pri}
		Let $G$ be a group, and $H$ a non-trivial normal subgroup of $G$.
		\begin{enumerate}[(i)]
			\item For each $x\in G$ and $h \in H$, $o(xH) \mid o(xh)$,
			\item $\frac{\psi(H)-|H|}{\sg}  +  \OO(\frac{G}{H})\le\OO(G)$. In particular, $\OO(\frac{G}{H})<\OO(G)$.
		\end{enumerate}
	\end{lemma}

	\begin{lemma}\cite[Theoren 2.5]{Pot}\label{Pot}
		Let $\theta$ be an automorphism of $G$ and $P$ a Sylow $p$-subgroup of $G$. If $r(G,\theta)>\frac{2}{p+1}$, then $P$ is normal in $G$.
	\end{lemma}

	\begin{lemma}\label{strkey}
		Let $K$ be a minimal normal subgroup of a finite group $G$. If $|K|$ is odd and $G/K\cong C_2\times C_2$, then $|K|=p$, for some prime $p$ and $G\cong C_2\times C_{2p}$ or $G\cong D(C_{2p})$. 
	\end{lemma}
	\begin{proof}
		Note that $G$ is solvable, so $K$ is an elementary abelian $p$-subgroup, for some odd prime $p$. Since $G$ is not a Frobenius group, there exists a subgroup $M$ of $G$ of index $2$, where $M$ is not a Frobenius group. Let $Q$ be a Sylow 2-subgroup of $M$.  As $C_K(Q)\le K$ is the normal Sylow $p$-subgroup of ${\bf Z}(M)$, we get that $C_{K}(Q)=K$, and so $G\cong C_2\times (K\rtimes C_2)$ or $G\cong C_2\times C_2\times K$. In both cases, since $K$ is minimal normal, $|K|=p$, as wanted.
	\end{proof}
	\begin{lemma}
		Let $\frac{m}{n}$ be a rational number, where $(m,n)=1$,
		and $m$ is even. There exists no finite group $G$ such that $\OO(G)=\frac{m}{n}$.
	\end{lemma}
	\begin{proof}
		Assume that there exists a finite group $G$ such that $\OO(G)=\frac{m}{n}$.
		Then by the definition of $\OO(G)$, $\psi(G)= \frac{m}{n}\sg$. Now by the fact that $\psi(G)$ is always odd, we get a contradiction.
	\end{proof}
	
	Therefore, $\OO(G)$ can not be equal to  $2.4$, $2.8$,  $3.6$, or any even integer, for any finite group $G$.

	\begin{lemma}\label{nil}
		
		Let $G$ be a finite nilpotent group. Then the following statements hold.
		
		\begin{enumerate}[(i)]
			\item If $|\pi(G)|>1$, then $\OO(G)\ge \OO(C_{6})=3.5$, and if $G\ncong C_{6}$, then $\OO(G)>4$.
			\item If $\sg>5$ is odd and $G$ is not a $3$-group, then $\OO(G)\ge\OO(C_{5}^{2})=4.84$.
		\end{enumerate}
		   
	\end{lemma}
	\begin{proof}
		($i$) By \cite[Lemma 1.1]{hlm}, if $\pi(G)=\{p_1,p_2,..., p_n\}$, then $\OO(G)=\prod\limits_{1\le i\le n}{o({{P}_{i}})}$, where $P_i\in \Syl_{p_i}(G)$. If $\sg$ has a prime divisor $p\ge5$, then $G$ has a quotient isomorphic to its Sylow $p$-subgroup, say P, and since $\OO(P)\ge \OO(C_{p})>4$, the statement holds. Hence, we may assume that $\pi(G) = \left\{2,3\right\}$. If $G\ncong C_{6}$, then $12\mid \sg$ or $18\mid|G|$. Therefore, $\OO(G)\geq \OO(C_2^2) \OO(C_3)> 4$.  
		
		($ii$)  If $\sg$ has a prime divisor $p\ge7$, then $G$ has a quotient isomorphic to its Sylow $p$-subgroup, say $P$, and since $\OO(P)\ge \OO(C_{p})>6>\OO(C_{5}^{2})$, the statement holds. Otherwise, $\pi(G)=\left\lbrace 5 \right\rbrace $, or $\pi(G)=\left\lbrace 3, 5 \right\rbrace $. In the first case, $G$ is a $5$-group and as $\sg>5$, the statement holds. In the second case, since $G$ has a factor, isomorphic to $C_{15}$, we have $\OO(G)\geq \OO(C_{15})=9.8>\OO(C_{5}^{2})$.
	\end{proof}

	\begin{lemma}\label{1,2}
		
		Let $G$ be a finite group and $H$ a subgroup of $G$.
		\begin{enumerate}[(i)]
			\item 	If $H$ is of index $2$, and $\OO(H)> 3.6$, then $\OO(G)> 2.8$. 
			\item	If $H$ is a normal subgroup of index $3$ and $\OO(H)> 2.4$, then $\OO(G)> 2.8$.
		\end{enumerate}
	\end{lemma}
	\begin{proof}
		($i$) As $G=H\cup xH$, for some $x\in G\setminus H$, we have,  $\OO(G)\geq \frac{\psi(H)+2|H|}{2|H|}=\frac{\OO(H)}{2}+1> 2.8$. 	
		The proof of ($ii$) is similar to ($i$). By Lemma \ref{pri}, $\OO(G)\geq \frac{\psi(H)+6|H|}{3|H|}=\frac{\OO(H)}{3}+2> 2.8$, as wanted.  
	\end{proof}
	
	\begin{lemma}\label{FB}
		Let $M$ be a subgroup of index $2$ of a finite group $G$. If $n_{2}(G\setminus M)>2|M|/3=\sg/3$, then $M$ is nilpotent.
	\end{lemma}
	\begin{proof}
		Let $x\in G\setminus M$ be an involution and $\tau_{x}\in \Aut(M)$, such that for each $m \in M$, $\tau_{x}(m)=m^{x}$. Note that $\tau_{x}$ maps $m\in M$ to its inverse if and only if $o(xm)=2$. Since $n_{2}(G\setminus M)>2|M|/3$, we get that $\tau_{x}$ maps more than $\frac{2}{3}$ of the elements of $M$ to their inverses. As, $r(M,\tau_{x})>\frac{2}{3}\ge \frac{2}{p+1}$, for any prime number $p$, by Lemma \ref{Pot}, $M$ is nilpotent.
	\end{proof}

	\section{main results}

	\begin{lemma}\label{ashkan}
		Let $G$ be a finite group and $m>1$ be an odd integer.
		\begin{enumerate}[(i)]
			\item If $|G|=2m$, then $\OO(G)\ge\OO(D(C_{5}\times C_{5}))=3.42$, unless $G\cong D(C_3^{k})$, for some integer $k$, or $G\cong D_{10}$.
			
			\item If $|G|=4m$, then $\OO(G)>3$, unless $G\cong A_4$ or $D_{12}$.
		\end{enumerate}
	\end{lemma}
	\begin{proof}
		($i$) Let  $N$ be  the Hall $2'$-subgroup of $G$, and $G=N\cup xN$, for some involution $x\in G\setminus N$. Note that $N$ is a group of odd order, so by \cite[Lemma 1.1]{hlm}, $\OO(N)\ge\OO(C_3)=\frac{7}{3}$. Note that for any $n\in N$, $o(xn)=2$ if and only if $n^{x}=n^{-1}$. Hence, by \cite[Lemma 10.4.1]{extra}, $n_{2}(xN)=n_{2}(G\setminus N)$ is a divisor of $|N|$.
		If $n_2(G\setminus N)\le \frac{|N|}{3}=\frac{\sg}{6}$, then	$\psi(xN)\ge 2 n_{2}(xN)+6(|N|-n_{2}(xN))\ge 2\frac{\sg}{6}+6 \frac{\sg}{3}=\frac{7}{3}\sg$. Thus,  $\OO(G)\ge \frac{\OO{(N)}}{2}+\frac{7}{3}\ge \frac{7}{6}+\frac{7}{3}=3.5$. If $n_2(G\setminus N)> \frac{|N|}{3}$, we see that $n_{2}(xN)=|N|$. So $\langle x \rangle$ acts Frobeniusly on $N$. Therefore, $N$ is abelian. First assume that $N$ is a 3-group.
		If exp($N$)$>3$, then $N$ has a quotient isomorphic to $C_9$, and so $\OO(N)\ge\OO(C_9)>6$. Hence, $\OO(G)= \frac{\OO(N)}{2}+1> 3+1=4$. Otherwise, exp($N$)$=3$, and so $G\cong D(C_3^{k})$, for some $k$, as wanted.
		If $N$ is not a $3$-group and $m>5$, then by Lemma \ref{nil}($ii$), $\OO(N)\ge\OO(C_{5}\times C_{5})=4.84$, and so $\OO(G)=\frac{\OO(N)}{2}+1\ge 2.42+1=3.42$. Finally, for $m=5$, $G\cong D_{10}$, and $\OO(D_{10})=3.1$, as desired.
		
		\bigskip
		($ii$)	Let $Q \in$ Syl$_2(G)$. If $Q$ is cyclic, then $G\cong N \rtimes Q$, where $N$ is the  Hall $2'$-subgroup of $G$. Since $N$ is of odd order, $\OO(N)\ge\OO(C_3)=\frac{7}{3}$. By Lemma \ref{pri}, $\OO(G) \ge \frac{\OO(N)}{4}+\frac{10}{4}\ge \frac{7}{12}+\frac{10}{4}=\frac{37}{12}>3$.
		Now,  assume that $Q$ is not cyclic. Let $m=3$. For the nilpotent groups of order 12, by Lemma \ref{nil}, $\OO(G)>3$. Looking at the non-nilpotent groups of order $12$ with non-cyclic Sylow $2$-subgroups, we see that the only possibilities are $D_{12}$ and $A_4$.  From now on,  we assume that $m\ge5$, and by induction on $m$, we prove that $\OO(G)>3$.
		
		Note that by \cite[Theorem C]{hlm}, if $G$ is non-solvable, then $\OO(G)>\OO(A_{5})>3.5$, and the result holds. If $G$ is solvable, there exists a maximal normal subgroup $M$ of index $p$, for some prime $p$. From the fact that $p-1+\frac{1}{p}=\OO(C_p)=\OO(\frac{G}{M})<\OO(G)$, if $p\ge5$, we get that $\OO(G)>3$. In the sequel, we consider $p\in \left\{2, 3\right\}$.
		Checking by \textbf{GAP}, we get that the statement holds for groups of order $36$ and so in the following we may assume that $|G|\ne36$. Now we consider the following cases.
		
		\bigskip
		
		{\bf 	Case 1)} If  $p=2$, then $G$ has a normal $2$-complement, say $K$.
		Let $K_0 \le K$ be a minimal normal subgroup of $G$. If $K_{0}<K$, we have $|G:K_{0}|=4m'$, where $m'\ge 3$. If $m'\ge 5$, by the induction hypothesis, $\OO(G)>\OO(G/K_{0})>3$. Let $m'=3$. So, $|G:K_{0}|=12$. Note that since $A_{4}$ does not have a normal $2$-complement, $G/K_{0}\ncong A_{4}$. If $G/K_{0}\ncong D_{12}$, then by the discussion we had for groups of order $12$, $\OO(G)>\OO(G/K_{0})>3$. Otherwise, $G/K_{0}\cong D_{12}$. Let $M$ be a subgroup of $G$ such that $M/K_{0}\cong C_{6}$. So,  $\OO(M)> \OO(C_6)=3.5$. If $n_2(G\setminus M)\leq |G|/3$, then $\OO(G)\geq\frac{ \psi(M)+ 2|G|/3+6|G|/6}{|G|}\geq \frac{7}{4}+\frac{5}{3}>3 $. Otherwise, $n_2(G\setminus M)> |G|/3$, then by Lemma \ref{FB}, $M$ is nilpotent.   Thus, by  Lemma \ref{nil}($i$),  $\OO(M)>4$. Therefore, $\OO(G)\ge\frac{\OO(M)}{2}+1>3$, as desired. 
		Now assume that $K_{0}=K$.  Therefore, by Lemma \ref{strkey}, $G\cong D(C_{2q})$ or $G\cong C_{2q}\times C_2$, where $q>3$. In the first case, $\OO(G)=\frac{\OO(C_{q}\times C_{2})}{2}+1>3$. In the latter case, by Lemma \ref{nil}($i$), $\OO(G)>3$.
		
		\bigskip
		
		{\bf 	Case 2)} If $p=3$, then $G=M\cup xM\cup x^2M$, for some $x\in G\setminus M$. On the other hand, by the induction hypothesis, $\OO(M)> 3$. So, $\OO(G)\ge \frac{\OO(M)}{3}+2>3$.
		
	\end{proof}

	\begin{lemma}\label{Stipe}
		Let $G$ be a finite group with $|\pi(G)|\ge2$. Then $\OO(G)\ge\OO(S_{3})=\frac{13}{6}$, and equality holds if and only if $G\cong S_{3}$.
	\end{lemma}
	\begin{proof}
		For the groups of order $6$, the statement holds. So we may assume that $\sg\ge10$. If $G$ is of odd order, as exp($G$)$\ge 3$, $\OO(G)\ge3-\frac{2}{\sg}>2.86>\frac{13}{6}$. Now assume that $G$ is of even order. If $n_{2}(G)>\frac{2}{3}\sg-1$, then since $r(G,1_{G})>\frac{2}{3}$, by Lemma \ref{Pot}, the Sylow $2$-subgroup of $G$, say $P$, is normal in $G$. Now because $|G:P|<\frac{\sg}{n_{2}(G)+1}<\frac{3}{2}$, $G$ is a $2$-group, a contradiction. Hence, $n_{2}(G)\le\frac{2}{3}\sg-1$. As, $\psi(G)\ge 1+2\nn+3(\sg-\nn-1)\ge\frac{7}{3}\sg-1$, we have $\OO(G)\ge\frac{7}{3}-\frac{1}{\sg}>\OO(S_{3})$, since $\sg\ge10$.
	\end{proof}
	
	Previous lemma implies that $\OO(G)\le S_{3}$, leads to $G\cong S_{3}$, or $G$ is a $2$-group. Now by \cite[Corollary]{new}, we get that $G$ is an elementary abelian $2$-group, if $G \ncong S_{3}$, which is a new proof for \cite[Theroem A]{hlm}.

	\bigskip
	
	The following lemma is obtained by \cite[Theorem 1.2]{t}, but we prove it  without referring  to that result.

	\begin{lemma}\label{less than 2.4}
		If $\OO(G)<2.4$, then $G$ is isomorphic to one of the following groups:
		\begin{enumerate}[(i)]
			\item $C_{3}$;
			\item $D_{8}$;
			\item $S_{3}$;
			\item $D(C_{3}\times C_{3})$;
			\item an elementary abelian $2$-group.
		\end{enumerate}
	\end{lemma}
	\begin{proof}
		
		We note that the result holds when $\sg\le10$. So we may assume that $\sg\ge11$. If $\sg$ is odd, as $exp(G)\ge3$, it is easy to check that $\OO(G)\ge3-2/\sg>\OO(C_{3}\times C_{3})>2.7$, a contradiction. Now assume that $\sg$ is even. Obviously the statement holds for elementary abelian $2$-groups and if $G$ is a $2$-group which is not elementary abelian, then by \cite[Corollary]{new}, $\nn\le \frac{3}{4}\sg-1$. Therefore, $\OO(G)\ge\frac{1+2\nn+4(\sg-\nn-1)}{\sg}=2.5-\frac{1}{\sg}$, and since $\OO(G)<2.4$, we have $\sg\le8$,  a contradiction. In the sequel, suppose that $G$ is not a $2$-group. If there exists a non-normal Sylow $p$-subgroup $P$ of $G$, for some odd prime $p$, then by considering $1_{G} \in \Aut(G)$ in Lemma \ref{Pot}, we have $r(G,1_{G})\le\frac{2}{p+1}$. It follows that $\nn\le\frac{2}{p+1}\sg-1$. So, $\psi(G)\ge 1 + 2\nn+3(\sg-\nn-1)>(3-\frac{2}{p+1})\sg-1$. Therefore, $\OO(G)\ge 2.5-\frac{1}{\sg}>2.4$, since $p\ge3$ and $\sg\ge11$, a contradiction. Therefore, $G\cong H\rtimes Q$, where $H$ is the Hall $2'$-subgroup of $G$, and $Q$ is a Sylow $2$-subgroup of $G$. Note that as we discussed in Lemma \ref{ashkan}, if $|Q|\le4$, then $G \cong D(C_{3}\times C_{3})$, as desired. Let $|Q|\ge8$ and $M$ be a subgroup of index 2 of $G$. If $n_{2}(G\setminus M)\le\frac{\sg}{3}$, as it was discussed multiple times, $\OO(G)\ge\frac{\OO(M)}{2}+\frac{4}{3}$. Since $|M|$ has at least two prime divisors, by Lemma \ref{Stipe}, $\OO(M)>\frac{13}{6}$, and so $\OO(G)>2.4$, which is a contradiction. Therefore, $n_{2}(G\setminus M)>\frac{\sg}{3}$, and by Lemma \ref{FB}, it follows that $M$ is nilpotent. Now Lemma \ref{nil}($i$) implies that $\OO(M)> 4$. Hence, $\OO(G)\ge\frac{\OO(M)}{2}+1>3 $, a contradiction.

	\end{proof}
	
	Note that the previous lemma implies that $\OO(G)<2$ if and only if $G$ is elementary abelian.
	\begin{remark}
	Throughout this paper, for simplicity we say $G$ is a $\star$-group, if $G$ is isomorphic to one of the following groups:
	\begin{enumerate}[\itshape(i)]
		\item $D_{8}\times D_{8}$;
		\item $D(A)$, where $A$ is an abelian $2$-group;
		\item $H(r)$, for some integer $r$;
		\item  $S(r)$, for some integer $r$.
	\end{enumerate}
\end{remark}

\begin{lemma}\label{2-groups}
	Let $G$ be a $2$-group with $\OO(G)<2.8$. Then $\nn>\frac{3}{5}\sg-2$, if $\sg\ge16$.
\end{lemma}
\begin{proof}
	On the contrary assume that $\nn\le\frac{3}{5}\sg-2$. So, $\OO(G)\ge\frac{1+2\nn+4(\sg-\nn-1)}{\sg}\ge2.8+1/\sg>2.8$, a contradiction.
\end{proof}

\begin{lemma}\label{the 2-groups}
	Let $G$ be a $\star$-group, and $\OO(G)<2.8$. Then $G$ is isomorphic to one of the following groups:
	\begin{enumerate}[(i)]
		\item $H(2)$;
		\item $S(2)$;
		\item $D(C_{4}\times C_{4})$;
		\item $D(C_{4}\times C_{2}^{k})$, for some integer $k\ge0$.
	\end{enumerate}
\end{lemma}

\begin{proof}
	We consider each $\star$-group separately. Note that $\OO(D_{8}\times D_{8})=\frac{183}{64}>2.8$. If $G\cong D(A)$, where $A$ is an abelian $2$-group, then $\OO(G)=\frac{\OO(A)}{2}+1$, and since $\OO(G)<2.8$, if follows that $\OO(A)<3.6$. Since $\OO(C_{8})>\OO(C_{4}^{3})>\OO(C_{4}\times C_{4} \times C_{2})>3.6$, we get that $A$ is isomorphic to $C_{4}\times C_{4}$ or $C_{4}\times C_{2}^{k}$, for some integer $k\ge0$. If $G\cong H(r)$, then $|H(r)|=2^{2r+1}$, and $n_{2}(H(r))=2^{2r}+2^{r}-1$. Now by Lemma \ref{2-groups}, $r\le 2$. If $G\cong S(r)$, then $|S(r)|=2^{2r+1}$, and we can see that $n_{2}(S(r))=2^{2r}+2^{r}-1$, which implies that $r\le 2$, by Lemma \ref{2-groups}. Note that $S(1)\cong H(1) \cong D_{8}$.
	
\end{proof}

Before we classify $2$-groups with $\OO(G)<2.8$, we take a close look at the bellow theorem, which is proved by Wall in \cite[Pages 261-262]{new}. It is the key to classify such $2$-groups.
\begin{Thm}[C. T. C. Wall, \cite{new}] \label{Wall}
	Let $G$ be a finite group. If $\nn>\frac{1}{2}\sg-1$, then $G$ is isomorphic to one of the following groups:
	\begin{enumerate}[\itshape(a)]
		\item $D(A)$, where $A$ is an abelian group;
		\item $D_{8}\times D_{8}\times C_{2}^{k}$, for some integer $k\ge 0$;
		\item $H(r)\times C_{2}^{k}$, for some integers $r$ and $k\ge 0$;
		\item  $S(r)\times C_{2}^{k}$, for some integers $r$ and $k\ge 0$.
	\end{enumerate} 
\end{Thm}

\begin{theorem}\label{classify}
	Let $G$ be a $2$-group with $\OO(G)<2.8$. Then $G$ is isomorphic to one of the following groups, where $k\ge0$ is an integer:
	\begin{enumerate}[(i)]
		\item $C_{4}$;
		\item $D_{8}\times C_{2}^{k}$;
		\item $D(C_{4}\times C_{4}) \times C_{2}^{k}$;
		\item $(D_{8}* D_{8})\times C_{2}^{k}$;
		\item $S(2)\times C_{2}^{k}$;
		\item an elementary abelian $2$-group.
	\end{enumerate}
\end{theorem}
\begin{proof}
	It is easy to check that the statement holds when $\sg<16$. Assume that $\sg\ge16$. If $G$ is abelian and $G$ has a factor isomorphic to $C_{4}\times C_{2}$ or $C_{8}$, then $\OO(G)>\OO(C_{4}\times C_{2})=2.875>2.8$, a contradiction. Obviously the result holds for elementary abelian $2$-groups. Hence, assume that $G$ is non-abelian. By Lemma \ref{2-groups}, we get that $\nn>\frac{3}{5}\sg-2$, and since $\sg\ge16$, $\nn>\frac{1}{2}\sg-1$. Now by the above theorem, we see that $G$ is isomorphic to one of the following $2$-groups:
	
	\begin{enumerate}[\itshape(1)]
		\item $D(A)$, where $A$ is an abelian $2$-groups;
		\item $D_{8}\times D_{8}\times C_{2}^{k}$, for some integer $k\ge 0$;
		\item $H(r)\times C_{2}^{k}$, for some integers $r$ and $k\ge 0$;
		\item $S(r)\times C_{2}^{k}$, for some integers $r$ and $k\ge 0$.
	\end{enumerate}
	
	In case \textit{(1)}, the statement holds by Lemma \ref{the 2-groups}.
	In other cases, let $E$ be the direct factor isomorphic to $C_{2}^{k}$, and since $\OO(G/E)\le\OO(G)<2.8$, again by Lemma \ref{the 2-groups} and some easy calculations we get the result.
\end{proof}

{\bf Proof of the \hyperref[A]{Main Theorem}.} 
Easily we can see that, if $G$ is isomorphic to one of the groups listed in the \hyperref[A]{Main Theorem}, then $\OO(G)<2.8$. Assume that $G$ is  a counterexample of minimal order. Remark that as $\OO(G)<\OO(A_{5})$, by \cite[Theorem C]{hlm}, $G$ is solvable.  Let $N$ be a minimal normal $q$-subgroup of $G$, for some prime $q$. As $\OO(G/N)< \OO(G)$ and there is no cyclic counterexample, $G/N$ satisfies the hypothesis of the theorem. So we consider each possibility for $G/N$  separately: 

\bigskip

{($i$)} Let $G/N\cong D_{12}$.

In this case, $G$ has a subgroup of index $2$, say $M$, such that $M$ has a quotient isomorphic to $C_6$. So,  $\OO(M)> \OO(C_6)=3.5$. If $n_2(G\setminus M)\leq |G|/3$, then $\OO(G)\geq\frac{ \psi(M)+ 2|G|/3+4|G|/6}{|G|}> \frac{7}{4}+\frac{4}{3}>2.8$, a contradiction. Hence, $n_2(G\setminus M)> |G|/3$, and by Lemma \ref{FB}, $M$ is nilpotent.   So by Lemma \ref{nil},  $\OO(M)> 4 $, implying a contradiction by Lemma \ref{1,2}.

\bigskip

{\bf ($ii$)} Let $G/N\cong S_4$.

In this case, by Lemma \ref{pri}, $\OO(G)\geq 2.75+\frac{\OO(N)}{24}$. Since $N$ is non-trivial, we get that $\OO(N)\geq 1.5$, hence $\OO(G)> 2.8$, a contradiction. 

\bigskip

{($iii$)} Let $G/N\cong C_3^a$, where $a\in \{1,2\}$.

In this case,  let $M$ be a normal subgroup of $G$ of index $3$. By Lemma \ref{1,2}, $\OO(M)< 2.4$. Now by Lemma \ref{less than 2.4}, $a=1$ and the only possibilities for $M$ are $C_3$ and $C_2^k$, for some integer $k$. In the first case, as $\OO(C_9)>6$, $G\cong C_3^2$, we get a contradiction. In the second case, $G$ is a group of order $3\cdot2^k$ and Fitting lemma implies that $N=C_{N}(P)\times [P,N]$, where $P\in \Syl_3(G)$.   By the fact that $N$ is a minimal normal subgroup of $G$ and $C_{N}(P)\triangleleft G $, we get that either $G$ is a  Frobenius group described in Case ($iv$), which is impossible as $G$ is a counterexample, or $G=P\times N$, a contradiction by Lemma \ref{nil}($i$).

\bigskip

{\bf ($iv$)} Let $G/N\cong C_2^{2k}\rtimes C_3$, be a Frobenius group,  for some integer $k$.

If $N$ is a $2$-group, then $N\le {\bf Z}(P)$, where $P$ is the Sylow $2$-subgroup of $G$.
Note that by Lemma \ref{1,2},  $\OO(P)< 2.4$, and by Lemma \ref{less than 2.4}, we get that $P\cong D_{8}$ or $P\cong C_2^a$, for some integer $a$ . Note that as $\Aut(D_{8})\cong D_{8}$, the first case implies the nilpotency of $G$, a contradiction by Lemma \ref{nil}. Therefore, $P\cong C_2^a$. 
First, assume that $N\leq {\bf Z}(G)$, then $|N|=2$.
Thus, $n_{6}(G\setminus P)\ge n_{3}(G\setminus P)$, and $$\OO(G)\geq \frac{\psi(P)+3(|G|-|P|)/2+6(|G|-|P|)/2}{3|P|}=\frac{\OO(P)}{3}+ 3> 3,$$ a contradiction. So, $N\cap {\bf Z}(G)=1$ and we get that the Sylow $3$-subgroup of $G$ acts on $N$ and $G/N$, Frobeniusly. Therefore, $G$ is a Frobenius group with an elementary abelian $2$-group as its kernel, which is the group described in Case ($iv$), a contradiction. 
Now, we assume that $|N|$ is odd.     In this case, using Lemma \ref{1,2}, $\OO(M)< 2.4$, where $M$ is an index $3$ normal subgroup of $G$. So, by Lemma \ref{less than 2.4},  $M\cong D( C_{3}\times C_{3})$, which implies $2k=1$, a contradiction. 

\bigskip  

{\bf ($v$)} Let $G/N$ be a $2$-group.

  Note that by Theorem \ref{classify}, $G$ is not a $2$-group. Now by Lemma \ref{ashkan}, $|G/N|\ge8$. We claim that $G/G'N$ is an elementary abelian $2$-group. Otherwise, there exists a normal subgroup $H$ of $G$ such that $G/H\cong C_{4}$,  $\OO(G)\geq \frac{\psi(H)+10|H|}{4|H|}=\frac{\OO(H)}{4}+2.5$, which implies that $\OO(H)<1.2$, a contradiction. So, $G/G'N$ is an elementary abelian $2$-group. First, let  $G'N=N$. Then every subgroup of order $2|N|$ is a normal subgroup of $G$.  Let $K$ be such a subgroup. Then the Sylow $q$-subgroup of ${\bf Z}(K)$ is normal in $G$ and by the fact that $N$ is a minimal normal subgroup, either $K\cong N\times M$, where $M$ is a group of order $2$, or $K$ is a Frobeniuos group.  If the first case occurs, by the minimality of $|G|$, $G/M$   is isomorphic to $D_{12}$. Note that since $G/N$ is completely reducible, $G$ splits on $M$. Therefore,  $G\cong C_2\times D_{12}$, which is  a contradiction, since $\OO(C_2\times D_{12})=\frac{73}{24}>3$.  So, $K$ is a Frobenius group. Since this holds for all subgroups of order $2|N|$, we get that $G$ is a Frobenius group and $G/N\cong C_2$, a contradiction. Therefore, $N<G'N$. Let $S$ be a normal subgroup of $G$ containing $N$, such that $G'N/S$ is a chief factor of $G$, isomorphic to $C_{2}$. Therefore, $G/S$ is a generalized extraspecial group (see \cite{gextra}), and by Theorem \ref{classify}, we conclude  that $G/S\cong D_8\times C_{2}^{k}$ or $G/S\cong D_{8}*D_{8}\times C_{2}^{k}$, for some $k\ge 0$. In the first case,  $G$ has a factor  $G/L$ isomorphic to $D_8$. Let $T/L$ be a subgroup of $G/L$ such that $T/L\cong C_4$.  Note that $N\le L$, hence, $\OO(L)>2$. Therefore,  $\OO(T)\geq\frac{\psi(L)+ 10|L|}{4|L|}=2.5+\frac{\OO(L)}{4}> 3$. If $n_2(G\setminus T)\le |G|/3$, $\OO(G)\geq \frac{\psi(T)+2|G|/3+ 4|G|/6}{|G|}=\frac{\OO(T)}{2}+ \frac{4}{3}> 2.8$, a contradiction. So, $n_2(G\setminus T)> |G|/3$, by Lemma \ref{FB},  $T$ is nilpotent. Then, by Lemma \ref{nil}($i$), $\OO(T)> 4$, which implies a contradiction by Lemma \ref{1,2}.
  In the second case, there exists a subgroup $W\le G$ of index $2$, where $W$ has a quotient, say $W/H$, isomorphic to $D_{8}*C_{4}$. Note that $\OO(D_{8}*C_{4})=\frac{47}{16}=2.9375$, and since $N\le H$, $\OO(H)>2$. By Lemma \ref{pri}, we get that $\OO(W)>3$. Now similar to the previous case, we get a contradiction. 
  
  \bigskip
  
  {($vi$)} Let $G/N \cong C_3^{k}\rtimes C_2$ be a Frobenius group,  for some integer $k$.

  By assumption, $|N|=q^{a}$, for some integer $a$. If $q$ is odd, then by Lemma \ref{ashkan}($i$), $q=3$ and $G\cong C_3^{a+k}\rtimes C_2$, which is a Frobenius group, a contradiction.
  So, $q=2$. Note that every subgroup  $M$ containing $N$ of order $3|N|$ is a normal subgroup of $G$. Hence, the Sylow $2$-subgroup of ${\bf Z}(M)$ is a normal subgroup of $G$, and as $N$ is minimal normal in $G$, either ${\bf Z}(M)=M$ or ${\bf Z}(M)=1$. In the first case, $M=N\times Q$, where $Q$ is the Sylow $3$-subgroup of $M$. By the above discussion, $G/Q$ is not isomorphic to the groups stated in Cases ($i$)-($v$), and since $4\mid \sg$, we get that $G/Q$ is not isomorphic to the group mentioned in Case ($vi$), a contradiction to the minimality of $\sg$. Whence, ${\bf Z}(M)=1$, which yields that $M$ is a Frobenius group. Thus, there is no element of order $6$ in $G$, implying that there is a  normal subgroup of $G$ of index $2$, say $T$,  which is a  Frobenius group and by the structure of the Frobenius groups, we get that $k=1$. So, Syl$_{2}(G)=\left\lbrace P_{1}, P_{2}, P_{3} \right\rbrace $, and $N=P_{i}\cap P_{j}$, for $1\le i<j\le 3$. If the Sylow 2-subgroups of $G$ are abelian, then $\bigcup\limits_{i=1}^{3}{P_{i}} \subset C_G(N)$, hence, $C_{G}(N)=G$, a contradiction. Whence, $\sg\ge24$. Sylow $2$-subgroups of $G$ are not abelian, so by \cite[Corollary]{new}, every Sylow $2$-subgroup of $G$ has at least $\sg/12$ elements of order $4$. On the other hand, we know that $\nnn=\sg-|P_{1}\cup P_{2} \cup P_{3}|=\sg/3$. Hence,
  $\psi(G)=1+3\nnn+4 n_{4}(G)+2(\sg-\nnn-n_{4}(G)-1) \ge\frac{17}{6}\sg-1$. Hence, $\OO(G)\ge\frac{17}{6}-\frac{1}{\sg}$, which implies that $|G|=24$, and so $G\cong S_4$, a contradiction. $\rule{0.67em}{0.67em}$

\hfil
\\
\begin{remark}
All the groups in the \hyperref[A]{Main Theorem} and their average orders are listed in the following tables:
\begin{table}[!htb]
	\begin{subtable}{.5\linewidth}\label{TA}
		\centering
		
		\begin{tabular}{|P{1.5cm}|P{1.5cm}|P{2.2cm}|}
			\hline
			Groups & $\psi(G)$ &  $\OO(G)$ \\ \hline
			$C_{3}^{k}\rtimes C_{2}$&  $5\cdot3^{k}-2$  & $2.5-1/3^{k}$ \\ \hline
			$C_{3}$ & $7$ & $7/3$ \\ \hline
			$C_{2}^{2k}\rtimes C_{3}$ & $2^{2k+3}-1$  & $8/3-1/(3\cdot 2^{2k})$ \\ \hline
			$D_{12}$ & $33$ &$2.75$ \\ \hline
			$C_{3}\times C_{3}$ &  $25$ & $25/9$ \\ \hline
			$S_{4}$ & $67$ &$67/24$ \\ \hline
		\end{tabular}	\caption{Non 2-groups with $\OO(G)<2.8$}
	\end{subtable}%
	\begin{subtable}{.5\linewidth}\label{TB}
		\centering
		\begin{tabular}{|P{2.5cm}|P{1.8cm}|P{2.1cm}|}
			\hline
			Groups & $\psi(G)$  &$\OO(G)$ \\ \hline
			$C_{2}^{k}$ & $2^{k+1}-1$ & $2-1/2^{k}$  \\ \hline
			$D_{8}\times C_{2}^{k}$& $5\cdot 2^{k+2}-1$ & $2.5-1/2^{k+3}$ \\ \hline
			$D(C_{4}\times C_{4}) \times C_{2}^{k}$ \newline $(D_{8}* D_{8})\times C_{2}^{k}$ \newline $S(2)\times C_{2}^{k}$ & \hfill \newline $11\cdot 2^{k+4}-1$ & \hfill \newline $11/4-1/2^{k+5}$ \\ \hline
			$C_{4}$ & $11$ &$2.75$  \\ \hline
		\end{tabular}\caption{$2$-Groups with o($G$)$<2.8$}
	\end{subtable}  
\end{table}

\end{remark}

{\bf Proof of the \hyperref[B]{Corollary}} 

{($a$)} Let $\OO(H)=\OO(S_{4})=\frac{67}{24}$. Then $|H|=24m$, for some odd integer $m$. Therefore, $H$ is not a $2$-group. Now using the \hyperref[A]{Main Theorem}, since $24\mid \sg$ and $\OO(H)<2.8$, we get that $G\cong S_{4}$. If $\OO(H)=\OO(G)$, where $G$ is one of the non $2$-groups mentioned in ($a$), similarly we get the result. If $\OO(H)=\OO(D_{8}\times C_{2}^{k})=(5\cdot 2^{k+2}-1)/2^{k+3}$, for some integer $k$, then $\sg=2^{k+3}m$, for some odd integer $m$. If $m\ge3$, then $H$ is isomorphic to one of the groups in \hyperref[TA]{Table A}. However, $(5\cdot 2^{k+2}-1)/2^{k+3}$ is not equal to the average orders in \hyperref[TA]{Table A}, a contradiction. So, $m=1$, and $H$ is a $2$-group of order $2^{k+3}$ with $\OO(H)<2.8$. In \hyperref[TB]{Table B}, we see that $H\cong D_{8}\times C_{2}^{k}$. Hence, $D_{8}\times C_{2}^{k}$ is characterizable by average order. The same discussion shows the charactrizablity of $C_{2}^{k}$ by average order.

\bigskip

{($b$)} Note that $\OO(D_{12})=\OO(C_{4})=2.75$. Now assume that $\OO(H)=\frac{11}{4}=2.75$. We get that $\sg=4m$, for some odd integer $m$. If $m=1$, as $\OO(C_{2}\times C_{2})=1.75$, $H\cong C_{4}$. Otherwise, since $\OO(H)<2.8$, by comparing the average orders of the groups in \hyperref[TA]{Table A}, we see that $H\cong D_{12}$.

\bigskip

{($c$)} Remark that $\OO(D(C_{4}\times C_{4}) \times C_{2}^{k})=\OO((D_{8}* D_{8})\times C_{2}^{k})=\OO(S(2)\times C_{2}^{k})$, for any integer $k$. Assume that $\OO(H)=(11\cdot 2^{k+4}-1)/2^{k+5}$, for some integer $k$. We see that $\sg=2^{k+5}m$, for some odd integer $m$.  If $m\ge3$, then as there is no group with such average order in \hyperref[TA]{Table A}, we get a contradiction. Therefore, $m=1$, and $H$ is a $2$-group of order $2^{k+5}$ with $\OO(H)<2.8$. Now \hyperref[TB]{Table B} shows that $H$ is isomorphic to one of three mentioned groups. $\rule{0.67em}{0.67em}$

\bigskip

According to the above corollary we pose the following two questions:
\begin{Question}
	What are the values of $n$ for which there exists a $n$-recognizable group by average order?
\end{Question}

\begin{Question}
	Is there any finite group $G$ such that there exists infinitely many non-isomorphic groups with average order $\OO(G)$?
\end{Question}

\bigskip
As an application,  by calculating the average orders in the \hyperref[A]{Main Theorem}, we get that other than $S_{4}$, there exists no group $G$, such that $\OO(G)$ lies in the interval $[\frac{67}{24},2.8]$. This is a partial answer to \cite[Conjecture 2.11]{t2}, about the density of $Im(\OO)=\left\lbrace \OO(G)| G \text{ is a finite group}\right\rbrace$. Now we see that $Im(\OO)$ is not dense in $[a,\infty)$, for any $a\le\frac{67}{24}$.

{\bf Conflict of Interest} he authors have no conflicts of interest to declare. All co-authors have seen and agree with the contents of the manuscript and there is no financial interest to report.

\end{document}